\documentclass[aos]{imsart}

\RequirePackage{amsthm,amsmath,amsfonts,amssymb,amscd,mathrsfs}
\RequirePackage[numbers]{natbib}
\RequirePackage[colorlinks,citecolor=blue,urlcolor=blue]{hyperref}

\newcommand{\R}{\mathbb{R}}

\newcommand{\G}{\mathscr{G}}

\newcommand{\X}{\mathscr{X}}

\newenvironment{enumerate*}

\startlocaldefs

\newtheorem{theorem}{Theorem}

\newtheorem{corollary}{Corollary}
\newtheorem{lemma}{Lemma}
\newtheorem{proposition}{Proposition}
\newtheorem{condition}{Condition}

\newtheorem{remark}{Remark}

\endlocaldefs

\begin{document}

\begin{frontmatter}
\title{On posterior consistency of data assimilation with Gaussian process priors: the $2D$-Navier-Stokes equations}
\runtitle{Posterior consistency for Navier-Stokes equations}

\begin{aug}

\date{\today}

\author[A]{\fnms{Richard}~\snm{Nickl}\ead[label=e1]{nickl@maths.cam.ac.uk}},
\and
\author[B]{\fnms{Edriss S.}~\snm{Titi}\ead[label=e3]{Edriss.Titi@maths.cam.ac.uk}}
\address[A]{Department of Pure Mathematics and Mathematical Statistics, University of Cambridge, Cambridge CB3 0WB, UK. \printead[presep={\ }]{e1}}

\address[B]{Department of Applied Mathematics and Theoretical Physics, University of Cambridge, Cambridge CB3 0WA, UK; and Department of Mathematics, Texas A\&M University, College Station, TX 77840, USA \printead[presep={\ }]{e3}}
\end{aug}

\begin{abstract}

We consider a non-linear Bayesian data assimilation model for the periodic two-dimensional Navier-Stokes equations with initial condition modelled by a Gaussian process prior. We show that if the system is updated with sufficiently many discrete noisy measurements of the velocity field, then the posterior distribution eventually concentrates near the ground truth solution of the time evolution equation, and in particular that the initial condition is recovered consistently by the posterior mean vector field. We further show that the convergence rate can in general not be faster than inverse logarithmic in sample size, but describe specific conditions on the initial conditions when faster rates are possible. In the proofs we provide an explicit quantitative estimate for backward uniqueness of solutions of the two-dimensional Navier-Stokes equations.

\end{abstract}

\end{frontmatter}

\section{Introduction}

High- and infinite-dimensional Bayesian methods have been increasingly popular in statistical inference problems arising with partial differential equations (PDEs) and related uncertainty quantification tasks, we selectively mention the contributions \cite{S10, CRSW13, HSV14, SS12, CLM16, BGLFS17} and references therein. Recent years have seen substantial progress in our theoretical understanding of the performance of such algorithms in non-linear settings, see \cite{N23} for an overview and many references. The results so far have covered a variety of prototypical examples ranging from basic steady state elliptic equations \cite{N20, GN20, AN19, NvdGW20, NW20} to X-ray-type problems \cite{MNP21, MNP21a, BN21, B21, StA22} and diffusion models \cite{NR20, GR22, K21, BGK20, N22a, HR22} where posterior distributions are obtained from updating a Gaussian process prior given noisy measurements of the solution of a PDE or SDE.

A particularly important and active application area of the Bayesian inference paradigm in PDE settings is the field of \textit{data assimilation} \cite{LSZ15, RC15, EVvL22}. For example in geophysical sciences, non-linear dynamical systems are used to model the atmosphere \cite{K03}, oceans \cite{B02}, turbulence \cite{MH12} and fluid flow \cite{CDRS09}. A Bayesian model for the initial conditions is then updated whenever new measurements are taken, and posteriors can be approximately computed by Monte Carlo and filtering methods. We refer to the recent monograph \cite{EVvL22}, especially its Part II, for an overview of a variety of concrete scientific application areas in the context of data assimilation. Beyond the setting of linear systems (e.g., \cite{KvdVvZ11, KVV17, R13}), the \textit{statistical validity} of such posterior based inferences in small noise or large sample size scenarios remains largely an open question. In the present article we study non-linear data assimilation problems arising with the periodic two-dimensional \textit{Navier-Stokes equations} as a paradigm for the underlying dynamics. This non-linear PDE provides the physical description of viscous flow in fluid mechanics and forms the mathematical foundation for the above-mentioned applications  in geophysical sciences. It also constitutes one of the key PDE examples for the Bayesian approach to data assimilation and inverse problems, see \cite{CDRS09} and \cite{S10}.

In the literature often reduced or approximate models (e.g., the Lorenz model) are used, see, e.g., \cite{K03, MH12, RC15, LSZ15, EVvL22}. To develop a general understanding we avoid such reductions as much as possible. The only substantial simplification we make is that we restrict to a two-dimensional state space -- this is simply because our theoretical development relies on PDE theory for the well-posedness of global solutions to Navier-Stokes equations which is not (yet!) known to be valid in dimensions higher than $2$. We also only consider  \textit{periodic} boundary conditions to streamline the exposition and to make the proofs accessible to a wider audience, but this restriction is not essential.

For $\Omega \subset \R^2$ a bounded domain, let $L^2=L^2(\Omega)^2$ denote the space of square Lebesgue-integrable vector fields $u: \Omega \to \R^2$. The (incompressible) Navier-Stokes equations postulate the evolution in time $t \in [0,T]$ of a divergence free velocity vector field $u(t) =u_\theta(t) \in L^2(\Omega)^2, t \in [0,T],$ solving the non-linear evolution equation
\begin{align} \label{nstokes0}
\frac{du}{dt} + \nu A u +B(u,u) &=f \\
u(0)&=\theta \notag
\end{align}
where $\theta$ is the initial condition of the system, $f$ a forcing term and $\nu>0$ a viscosity parameter. The partial differential operators $A= -P\Delta$ and $B=P[(u \cdot \nabla) u]$ include a projection $P$ onto the Hilbert subspace $H \subset L^2$ of divergence free vector fields subject to some side condition (e.g., to have mean zero in the periodic case), $\Delta$ is the Laplacian, and $(u\cdot \nabla) v$ is vector calculus notation for the vector field with entries
\begin{align}\label{bee}
[(u \cdot \nabla) v]_j  = \sum_{i=1}^2  u_i \frac{\partial v_j}{\partial x_i},~~~j=1,2.
\end{align}
The Navier-Stokes equations further include the gradient of a `pressure term' which however is in the kernel of the projection operator $P$ and so does not feature in equation (\ref{nstokes0}). We will explain this further when we review this PDE in detail below.

\smallskip

The task of data assimilation begins with the specification of an initial condition $\theta$. In absence of specific background knowledge in a given experimental setting, and in part to aid the computational tasks that follow, such initial conditions are often modelled by a Gaussian random field $(\theta(x): x \in \Omega)$ over the domain $\Omega$ -- see \cite{S10, RC15, LSZ15, EVvL22} and also specifically \cite{CDRS09} in the setting of the Navier-Stokes model. The law $\mathcal L(\theta) \equiv \Pi$ of this field in an appropriate function space plays the role of the prior in (infinite-dimensional) Bayesian statistics, and its covariance structure parallels the choices of penalty norms in `variational (optimisation based) data assimilation' -- see Remark \ref{map}. Under appropriate assumptions the solutions to (\ref{nstokes0}) are unique and the random initial condition $\theta=u(0)$ determines a complete stochastic \textit{forward model}, that is, a probability distribution on the states
$$\big(u_\theta(t,x): t>0, x \in \Omega\big),~\text{that solve}~(\ref{nstokes0})~ \text{with initial condition}~ u(0)=\theta \sim \Pi,$$
of the velocity fields at all points in space $\Omega$ and time $[0,T]$. Even though the initial condition follows a Gaussian distribution, due to the non-linearity of the Navier-Stokes system, the implied stochastic model for $u_\theta(t,\cdot)$ at times $t>0$ is \textit{not} Gaussian any longer. Nevertheless, this `forward model' can be updated via Bayes' rule after measurements at discrete points $(t_i, X_{ij})$ in $(0,T] \times \Omega$ are taken. We follow here the `Eulerian measurement' scheme relevant in fluid mechanics (see Sec.3 in \cite{CDRS09}) where noisy `regression' type measurements of the velocity field of the form
\begin{equation}\label{discretem}
Y_{ij} = u_\theta(t_i, X_{ij}) + \varepsilon_{ij},~~~\varepsilon_{ij} \sim N(0,I_{\R^2}),~i=1, \dots, m,~ j=1, \dots, n,
\end{equation}
are collected. If $\Pi$ denotes the prior probability measure induced by the Gaussian process model, the posterior distribution of the initial state $\theta$ is then given by
$$d\Pi(\theta|(Y_{ij}, X_{ij}, t_i)_{1\le i \le m, ~ 1\le j \le n}) \propto \exp \Big\{-\frac{1}{2}\sum_{i,j}|Y_{ij} - u_{\theta}(t_i, X_{ij})|_{\R^2}^2 \Big\} d\Pi(\theta),$$
where $u_\theta$ is the solution of (\ref{nstokes0}) corresponding to initial condition $u(0) = \theta$. Approximate computation of such non-Gaussian posterior distributions is possible using Monte Carlo methods, and we can then use PDE forward solvers or filtering methods to retrieve posterior inferences for $u_\theta$ at times $t>0$ and points $x \in \Omega$ -- see \cite{LSZ15, RC15, CRSW13, BGLFS17, EVvL22}. Of key importance is that -- in contrast to standard non-parametric regression techniques -- such posterior estimates $u_\theta$ are themselves solutions of a Navier-Stokes system so that the algorithmic outputs retain physical interpretation (e.g., incompressibility of the flow, $\nabla \cdot u_\theta(t)=0$ at all times $t>0$).

Such Bayesian methodology  naturally addresses all main tasks of data assimilation -- `prediction, filtering, smoothing and inversion', 
cf.~p.173 in \cite{RC15} and also \cite{LSZ15} for this terminology -- which essentially amount to providing estimates and credible regions for $u(t, \cdot)$ at all times $t$. The question we address here is whether the \textit{posterior distribution} of all the states $$\big(u_\theta(t,x): t\ge 0, x \in \Omega\big),~~~ \theta \sim \Pi(\cdot|(Y_{ij}, X_{ij}, t_i)_{1\le i \le m, ~ 1\le j \le n}),$$ is \textit{statistically consistent}, that is, whether it places its mass, measured in a suitable norm in function space, near the ground truth state $(u_{\theta_0}(t,x): t \ge 0, x \in \Omega)$ of the system generated by the actual (unobserved) initial condition $u(0)=\theta_0$, and with high probability under the law of the data. Such results validate Bayesian data assimilation algorithms in a scientifically desirable `objective', that is, \textit{prior-independent} or `frequentist' way, cf.~\cite{GV17}. To the best of our knowledge, no results of this type are known in the literature at the moment. We will show in Theorem \ref{mainstat} that posterior consistency indeed occurs for a flexible class of infinite-dimensional Gaussian process prior models for $\theta$, if sufficiently many measurements are taken within the given observation horizon. That is, the asymptotics is for $N=mn \to \infty$ with $T$ fixed, corresponding to statistically highly informative data. Of course the Bayesian method can be used to quantify uncertainty even for low signal to noise ratios, in which case the influence of the prior would remain visible and our `frequentist' guarantees only demonstrate that the uncertainty about the parameter is reduced by taking more measurements. 

Our proof strategy resembles the Bayesian forward model and is based on solving the hardest data assimilation problem -- inversion for the initial condition $u(0)$ -- first. Inference for all subsequent times will then be valid too by appealing to forward Lipschitz continuity with respect to the initial data of strong solutions of the (two-dimensional) Navier-Stokes equations. Considering $T>0$ to be fixed in the asymptotics is natural in this context: for $T \to \infty$ the Navier-Stokes dynamics stabilise near a finite-dimensional global attractor \cite{CF88, R01} and hence the long time behaviour cannot be expected to provide information about an infinite-dimensional parameter space of initial conditions. This is related to the inherent unpredictability of geophysical systems in large time horizons -- see the classical contribution of Lorenz \cite{L63} and also \cite{K03}.  We further show that even when $T$ is fixed, the `logarithmic' convergence rates we obtain for the inversion problem as $N \to \infty$ cannot be improved in general, but we also discuss a set of (strong) restrictions on the initial condition in terms of the spectrum of the Stokes operator where faster rates are possible. Our results hold for fixed positive viscosity $\nu>0$ corresponding to a model of viscous flow -- small viscosity limits with dominant non-linear term may lead to a different theory but are not investigated here.

The main ingredients of our proofs consist of a) statistical theory for Bayesian non-linear inverse problems developed recently (see \cite{N23}), b) results from the PDE analysis of $2D$ Navier-Stokes equations (e.g., \cite{CF88, R01}), and c) an explicit quantitative stability (inverse continuity) estimate for the forward map $\theta \mapsto u_\theta$ given in Theorem \ref{main1} below. The latter is inspired by old work on backward in time uniqueness of solutions to the $2D$ Navier-Stokes equations in \cite{BT73}, who gave explicit estimates on the difference between the initial state in terms of the difference of the final state of the strong solutions of the system. This implies the desired backward uniqueness, but further allows one to obtain a quantitative stability estimate. These ideas extend in principle to other dissipative dynamical systems such as reaction-diffusion equations $(\partial / \partial t) u - \Delta u = f(u)$ with appropriate non-linear $f$.

\smallskip

This paper is organised as follows: the main analytical results on the $2D$ Navier-Stokes equations will be given in Subsection \ref{nstokest}, while the theory for data assimilation with Gaussian process priors is developed in Subsection \ref{datsim}. Proofs can be found in Section \ref{proofs}.

\section{Main results}

\subsection{Forward and inverse stability in the $2D$ Navier-Stokes equations}\label{nstokest}

Throughout we denote by $\Omega =[0,2\pi]^2$ the two-dimensional flat torus, i.e., opposite endpoints are identified and all functions are periodic: $u(\cdot + 2\pi e_i)=u(\cdot)$ for $i=1,2$ where  $e_1=(1,0), e_2=(0,1)$ are the canonical basis vectors in the plane. We define $C^\infty(\Omega)$ as the space of infinitely differentiable periodic functions with fundamental periodic domain $\Omega$. We also require the usual $L^2(\Omega)$ spaces of square integrable functions for Lebesgue measure $dx$, as well as the Sobolev spaces $H^m(\Omega), m \in \mathbb N,$ of functions $f \in L^2(\Omega)$ whose (weak) partial derivatives up to order $m$ lie in $L^2(\Omega)$.  When considering two-dimensional vector fields $v=(v_1, v_2): \Omega \to \R^2$ with components $v_1,v_2$ lying in some function space $\X$, we will write $v \in \X^2$ -- or sometimes even only $v \in \X$ when no confusion may arise. The divergence operation $$\nabla \cdot v = \frac{\partial}{\partial x_1}v_1 + \frac{\partial}{\partial x_2}v_2$$ for smooth vector fields extends to a linear operation $\nabla \cdot$ in the sense of (periodic Schwartz) distributions. We can then define spaces of vector fields
\begin{equation}
H = \Big\{u \in L^2(\Omega)^2:  \nabla \cdot u =0, \int_\Omega u  =0 \Big\},
\end{equation}
as well as
\begin{equation}\label{V}
V = \Big\{u \in (H^1(\Omega))^2: \nabla \cdot u =0, \int_\Omega u  =0 \Big\},
\end{equation}
and equip these spaces with inner products $\langle \cdot, \cdot \rangle_H \equiv \langle \cdot, \cdot \rangle_{L^2}$ and $$\langle u, v\rangle_V \equiv \langle \nabla u, \nabla v \rangle_{L^2} = \sum_{i,j=1}^2 \int_\Omega \frac{\partial u_i(x)}{\partial x_j} \frac{\partial v_i(x)}{\partial x_j} dx,$$ where $\nabla$ is the gradient operator. The resulting norms are denoted by $\|\cdot\|_H, \|\cdot\|_V$, respectively. One can show that $H$ and $V$ arise as the closure in $L^2(\Omega)$ and $H^1(\Omega)$, respectively, of the divergence free, non-constant trigonometric polynomials  (cf.~p.236f.~in \cite{R01}).

\smallskip

We now review some standard facts from the theory of $2D$-Navier-Stokes equations -- see \cite{CF88} or \cite{R01} for classical references on this material. The `Helmholtz-Leray' $L^2$-projector $P: L^2(\Omega)^2 \to H$ extends to act on all Schwartz distributions. The \textit{Stokes operator} is then $A = -P\Delta$ where $\Delta = \nabla \cdot \nabla$ is the Laplacian. In the case of periodic boundary conditions one in fact has $A = - \Delta$ on its domain $$\mathcal D(A) \equiv H^2(\Omega)^2 \cap V,$$ (see (9.15) in \cite{R01}). The operator $A$ then has `graph norm' $$\|u\|_{\mathcal D(A)} \equiv \|Au\|_{L^2} \simeq \|u\|_{H^2},$$ (see (9.13) in \cite{R01}). Next, using appropriate versions of Sobolev inequalities (p.243 in \cite{R01} or Ch.6 in \cite{CF88}) and recalling (\ref{bee}), we can realise the bilinear form
\begin{equation}\label{B}
B(u,v) = P[(u \cdot \nabla) v], ~\text{ as }~B: V \times V \to V',
\end{equation}
where $V'$ is the topological dual space of $V$ with the usual dual pairing $\langle \cdot, \cdot \rangle_{V,V'}$ arising from the action of the $L^2$-inner product. In this notation we have from integration by parts
\begin{equation} \label{divA}
\|u\|_{V}^2 = \langle \nabla u, \nabla u \rangle_{L^2} = \langle -\Delta u, u\rangle_{V,V'} = \langle Au, u \rangle_{V,V'}, ~u \in V.
\end{equation}
Moreover (as in (6.17), (6.18) in \cite{CF88}, or Proposition 9.1 in \cite{R01}), for $u,v,w \in V$, one has
\begin{equation}\label{divB}
\langle B(u,v), w \rangle_{V,V'} = - \langle B(u,w), v \rangle_{V, V'},~\text{and thus}~\langle B(u,v), v \rangle_{V,V'} =0.
\end{equation}

\smallskip

Now let $\nu >0$ be a fixed viscosity constant and $f  \in V$ a forcing term. To expedite proofs we take $f$ to be time-independent but this is not necessary. We are interested in (spatially) periodic vector fields $u=u_\theta=(u_\theta(t,x): t \in (0,T], x \in \Omega)$ that solve the incompressible Navier-Stokes equations represented by the system of non-linear partial differential equations
\begin{align}\label{nstokp}
\frac{\partial}{\partial t} u - \nu \Delta u + (u \cdot \nabla) u &= f -\nabla p~~~\text{ on } (0,T] \times \Omega,  \\
\nabla \cdot u&= 0 ~~~\text{ on } (0, T] \times \Omega, \notag \\
\int_\Omega u(t, \cdot) &=0 ~~~\text {for all } t \in (0,T], \notag \\
u(0,\cdot) &= \theta ~~~\text{ on } \Omega, \notag
 \end{align}
where $\theta \in V$ is an initial condition and $p$ is a scalar pressure term.

A complete solution theory exists for this PDE in our two-dimensional setting. One can initially consider a type of \textit{weak solution} by taking $L^2$-inner products with $v \in V$ in (\ref{nstokp}). Using (\ref{B}), (\ref{divA}) and that $\int_\Omega \nabla p \cdot v = \int_\Omega p (\nabla \cdot v) =0$ by integration by parts, we see that in order to find such a weak solution, it suffices to find a solution $u \in V$ of the non-linear evolution equation in $V'$ given by
\begin{align} \label{nstokes}
\frac{du}{dt} + \nu A u +B(u,u) &=f \\
u(0)&=\theta. \notag
\end{align}
When this equation holds in the space $H$ (rather than just in $V'$), we speak of a `strong' solution. See Ch.9 in \cite{R01} or Ch.5 in \cite{CF88} for details. For the problem of recovering the initial condition studied below, we shall content ourselves with the information provided by the solutions of the reduced equation (\ref{nstokes}), but it is not difficult to see (from the Helmholtz decomposition theorem) that (\ref{nstokes}) and (\ref{nstokp}) are in fact equivalent equations.

To formulate an existence result for strong solutions of (\ref{nstokes}), consider function spaces $$L^p((0,T], \mathscr X) \equiv \Big\{u: (0,T] \times \Omega \to \R^2: \int_0^T \|u(t,\cdot)\|_\mathscr X^p dt <\infty \Big\},~~1 \le p <\infty,$$ with corresponding Bochner-integral norm for $\mathscr X$-valued maps, where $\mathscr X$ is a normed linear space of vector fields over $\Omega$ to be specified. Similarly we define the spaces $L^\infty((0,T], \mathscr X)$ and $C([0,T], \mathscr X)$ of time-bounded or -continuous $\mathscr X$ valued maps. 

The next result follows from existing theory for the two-dimensional periodic Navier-Stokes equations \cite{CF88, R01}; only the uniform in space-time boundedness of the solutions in (\ref{ubd}) -- which will be important in the statistical proofs to follow -- is not entirely standard and we therefore include a proof in Section \ref{appendix} below.

\begin{proposition}\label{classical}
Let $T>0$ and let $u(0) \in V$ satisfy $\|u(0)\|_{V} \le U$ for some $U>0$.

\smallskip

A) The two-dimensional periodic Navier-Stokes equations (\ref{nstokes}) have a unique strong solution $u \in C([0,T], V) \cap L^2((0,T], \mathcal D(A))$ with $du/dt \in L^2((0,T], H)$. There exists a constant $c_U\equiv c(U, \|f\|_{L^2}, \nu,T )<\infty$ such that
\begin{equation} \label{fwdreg}
\sup_{0 \le t \le T} \|u(t)\|_V + \int_0^T \|u(t)\|_{H^2(\Omega)}^2 dt \leq c_U.
\end{equation}
Moreover, for every $m>0$ there exists $c=c(m, \|f\|_{H^1}, \nu, T)>0$ such that we have
\begin{equation}\label{ubd}
\sup_{u(0) \in \mathcal D(A):~\|Au(0)\|_{L^2} \le m} ~\sup_{0 \le t \le T, x \in \Omega}|u(t,x)| \le c<\infty.
\end{equation}

B) If $v(0) \in V$ is another initial condition, then we have
\begin{equation}\label{fwdstab}
\sup_{0 \le t \le T} \|u(t)-v(t)\|_{L^2(\Omega)} \le K \|u(0) - v(0)\|_{L^2(\Omega)}
\end{equation}
for some constant $K=K(U, \|f\|_{L^2}, \nu, T)<\infty$.
\end{proposition}

Let us now turn to the inverse problem of solving for $u(0)$ from $u(t)$. It is known that under natural regularity hypotheses, the Navier-Stokes solution map $\theta = u(0) \mapsto u(t)=u_\theta(t)$ is analytic in time, see Theorem 12.2 in \cite{CF88}. Therefore we have `backward uniqueness' -- the solutions $u(t)$ determine their initial conditions $u(0)$. As this result is non-quantitative, it is of little use to reconstruct $u(0)$ from discrete noisy measurements of $u(t)$.  A main contribution of this article is the following explicit stability estimate for this forward map $\theta \mapsto u_\theta(t)$, to be used in the proofs of the statistical theorems below. It is inspired by backward uniqueness results \cite{BT73} for general non-linear parabolic equations in Hilbert space. Our proof implies two separate inequalities, of independent interest: the first is a global stability estimate with a logarithmic inverse modulus of continuity, corresponding to a `severely' ill-posed inverse problem. The second gives a potentially stronger Lipschitz stability estimate which however implicitly still depends on the initial conditions through a non-trivial constant $c_P$ which will be discussed in more detail in Remark \ref{stokspec}.

\begin{theorem}\label{main1}
Let $T>0$, and for initial conditions  $u(0), v(0) \in V$ such that $\|u(0)\|_{V} + \|v(0)\|_{V} \le U<\infty$, consider the corresponding strong solutions $u,v \in C([0,T], V)$, to the $2$-dimensional periodic Navier-Stokes equations (\ref{nstokes}).

\smallskip

A) There exists a constant $c_1$ depending only on $U, \nu, T, \|f\|_{L^2}$ such that $\sup_{0 \le t \le T}\|u(t)-v(t)\|_{L^2}<c_1$ and
\begin{equation}\label{logstab}
\|u(0)-v(0)\|_{L^2(\Omega)} \leq c_1  \Big(\log \frac{c_1}{\|u(t)-v(t)\|_{L^2}} \Big)^{-1/2},~~~\text{for every } t \in [0, T].
\end{equation}

B) Let further $0<c_P<\infty$ be a (`inverse Poincar\'e') constant such that
\begin{equation}\label{invpoinc}
\frac{\|u(0)-v(0)\|_{V}}{\|u(0)-v(0)\|_{L^2}} \le c_P .
\end{equation}
Then there exists a constant $c_2=c_2(U, \nu, T, \|f\|_{L^2})$ such that
\begin{equation}
\|u(0) - v(0)\|_{L^2(\Omega)} \le e^{c_2 c_P} \|u(t) - v(t)\|_{L^2(\Omega)},~~~~~~\text{for every } t \in [0, T].
\end{equation}
\end{theorem}

Versions of these inequalities where $\|u(t) - v(t)\|_{L^2(\Omega)}$ is replaced by its quadratic time average over intervals $[T_0,T]$ hold as well -- see Corollary \ref{usef} below. The proof of Theorem \ref{main1} extends to the case of bounded smooth domains $\Omega \subset \R^2$, if the spaces $V, H$ are appropriately defined (with Dirichlet boundary conditions, cf.~\cite{CF88}), after only minor technical modifications.

The growth of the constants in the preceding stability estimates is exponential in the time horizon $T$, as is not unexpected due do the limited time predictability of chaotic dynamical systems of even much simpler nature than the Navier-Stokes equations. We refer to the discussion in the introduction, the classical contribution by Lorenz \cite{L63}, as well as more generally \cite{K03} about limitations of forecasting geophysical systems. The previous stability estimates hence should be interpreted as informative in `moderate' time horizons $T$.

\begin{remark}[Inverse Poincar\'e inequality and the Stokes spectrum] \label{stokspec} \normalfont
Note that the ratio in (\ref{invpoinc}) is lower bounded by a fixed positive constant in view of the Poincar\'e inequality (p.292 in \cite{E10}), hence the terminology. Moreover, for $u(0)\neq v(0) \in V$, the ratio in (\ref{invpoinc}) is always finite, so a constant $c_P$ always exists in this case.  For the statistical results below we require however \textit{uniform} control of $c_P$ in the parameter space of admissible initial conditions $u(0),v(0)$. While one cannot in general expect that such a uniform `inverse Poincar\'e constant' exists, examples can be given where it holds. For instance suppose both $u(0), v(0) \in V$ have a \textit{finite} expansion in the $H$-orthonormal eigen-basis $(e_j: j \ge 0)$ of the Stokes operator $A$ (see (\ref{stokeef}) below), up to `frequency' $J$. Then
\begin{align*}
\frac{\|u(0)-v(0)\|^2_{V}}{\|u(0)-v(0)\|^2_{L^2}} = \frac{\sum_{j \le J} \lambda_j \langle e_j, u(0)-v(0) \rangle_{L^2}^2}{\sum_{j \le J} \langle e_j, u(0)-v(0) \rangle_{L^2}^2} \leq \lambda_J  \equiv c_P.
\end{align*}
Observe that one has the asymptotic distribution $\lambda_j \simeq j$ as $j \to \infty$, of the eigenvalues of the Stokes operator $A$ in $d=2$, see Proposition 4.14 in \cite{CF88}. This will permit the creation of `Stokes band-limited' models of initial conditions for which we can obtain `fast' (better than logarithmic) convergence rates in Theorem \ref{fastrates} below. Other sets of initial conditions could be conceived for which (\ref{invpoinc}) holds with a uniform constant $c_P$, but this is beyond the scope of the present paper.
\end{remark}
The stability estimate (\ref{logstab}) is sharp in the sense that the inverse modulus of continuity is attained on balls in $H^2$ for particular sets of motions of the Navier-Stokes equation -- at least up to the power of the logarithm. Inspection of the proof of (\ref{invest}) shows that the power of $\log$ in the lower bound could be made to approach $1/2$ if we consider initial conditions bounded in $H^1$ only, but we give a result under the stronger $H^2$-hypothesis relevant in the statistical theorems that follow.

\begin{theorem}\label{stabsharp}
There exists a sequence of initial conditions $u_j(0) \in C^\infty(\Omega)^2 \cap V$, $j \in \mathbb N,$ with corresponding strong solutions $u_j(t)$ to the periodic Navier-Stokes equations (\ref{nstokes}) on $\Omega$ with $\nu=1/2, f=0,$ such that
\begin{equation}\label{invest0}
\|u_j(0)\|_{H^2} \lesssim 1,~~\|u_j(0)\|_{L^2} \simeq j^{-2},~~\|u_j(t)\|_{L^2} \simeq e^{-j^2t}j^{-2},
\end{equation}
for all $t>0$. In particular, setting $v(0)=0$ and hence $v(t) \equiv 0$, we have for some $c'=c'(c,t)>0$ that
\begin{equation}\label{invest}
\|u_j(0)-v(0)\|_{L^2(\Omega)} \ge c' \frac{1}{\log \big(\frac{1}{\|u_j(t) - v(t)\|_{L^2}} \big)},~~\text{all } j \in \mathbb N, t>0.
\end{equation}
\end{theorem}

The exponential instability arises from a (linear, scalar) heat equation that can be `planted' within the set of solutions of Navier-Stokes equations for a specific set of initial conditions for which the non-linearity vanishes at all times. As long as $\nu>0$ and $f=0$ we conjecture that similar phenomena persist also with non-vanishing non-linear term by employing the spectral manifolds constructed in \cite{FS84} instead of the eigenfunctions of the Laplacian from the previous proof. In these examples, most of the information in the data is `at the beginning of time', reflecting the action of a smoothing semigroup generated by the Laplacian scaled by a non-vanishing viscosity $\nu$.

\subsection{Non-linear data assimilation with Gaussian process priors}\label{datsim}

We now consider data assimilation tasks arising with discrete observations from the Navier-Stokes equations and give statistical guarantees for Bayesian methodology proposed and developed for such problems in \cite{CDRS09} -- see also \cite{S10, RC15, LSZ15, EVvL22}. For $u=u_\theta$ a (strong) solution of the PDE (\ref{nstokes}) with unknown initial condition $u(0)=\theta \in V$, the statistical observations are assumed to consist of the random vectors $Z^{(N)}=(Y_i, t_i, X_i)_{i=1}^N$
\begin{equation}\label{model}
Y_{i} = u_\theta(t_i, X_i) + \varepsilon_{i},~~\varepsilon_{i} \sim^{i.i.d.} N(0,I_{\R^2}),~~i=1, \dots, N,
\end{equation}
with $(t_i, X_i)_{i=1}^N$ drawn iid from the uniform distribution $\lambda$ on $(T_0,T] \times \Omega$, independently of the Gaussian noise vectors $\varepsilon_i$. Here $(T_0,T]$ is the time horizon where measurements are taken and the (fixed) constants $T_0, T$ are such that $0 \le T_0 < T$. We include by convention the case $0<T_0=T$ of a single temporal measurement $(T_0, T]\equiv\{T\}$ is; in this case $\lambda = \delta_{T} \otimes \lambda_\Omega$ where $\delta_T$ is Dirac measure at $\{T\}$ and $\lambda_\Omega$ the uniform distribution on $\Omega$. The law on $(\mathbb R^2 \times (T_0,T] \times \Omega)^N$ of the data vector $Z^{(N)}=(Y_i, t_i, X_i)_{i=1}^N$ when $u_\theta$ arises from the initial condition $\theta$ will be denoted by $P_\theta^N$. One could  (as in (\ref{discretem})) consider distinct sample sizes for time and space measurements but we abstain from this as it facilitates the exposition. Let us emphasise that the `white' noise $\varepsilon_i$ in (\ref{model}) is purely of measurement error type (arising from the discretisation of $u_\theta$) and that we do not model explicit stochasticity in the Navier-Stokes dynamics itself, in particular we do not consider a SPDE model.

\smallskip

We assume the viscosity $\nu>0$ and forcing $f \in V$ to be known, possibly determined beforehand by independent experiments (theory for these distinct inverse problems could be developed as well, for $\nu$ for instance following ideas in \cite{GN20, N22a}). In contrast, and following common practice in data assimilation, we assume that the initial condition $\theta=u(0)$ of the system is \textit{unknown}. Inferences on the state $u_\theta(t,\cdot)$ of the system based on observations $Z^{(N)}$ need to incorporate this uncertainty. One systematic way to do this is to adopt a Bayesian approach and to model the initial condition $u(0)=\theta$ by a Gaussian random field over $\Omega$ -- see \cite{CDRS09} in the setting of Navier-Stokes equations and fluid mechanics specifically; and also \cite{S10, RC15, LSZ15, EVvL22} more generally. For the theory that follows we will employ the following assumption -- for standard notions of and background on Gaussian processes we refer the reader to \cite{GN16} and \cite{N23}.

\begin{condition}\label{gprior}
Consider a Borel probability measure $\Pi'$ on $V \cap H^2(\Omega)^2$ arising as the law of the centred Gaussian random vector field $(\theta'(x)=(\theta'_1(x), \theta'_2(x)): x \in \Omega)$ with reproducing kernel Hilbert space (RKHS) $\mathcal H$ continuously imbedded into $V \cap H^\alpha(\Omega)^2$ for some $\alpha \ge 2$. Then take as prior $\Pi=\Pi_N$ for $\theta$ the law of the rescaled random vector field $\theta = \theta'/N^{1/(2\alpha+2)}$.
\end{condition}

Examples of `base priors' $\Pi'$ with $\mathcal H = (V \cap H^\alpha(\Omega)^2, \|\cdot\|_{H^\alpha(\Omega)^2})$ for any $\alpha \ge 2$ can be easily constructed: for instance we could define the prior for $\theta'$ immediately as a Gaussian series expansion (e.g., (B.1) in \cite{N23}) for the (real parts of the) $H$-orthonormal eigenfunctions
\begin{equation}\label{stokeef}
e_k \propto (k_2,-k_1) e^{ik.(\cdot)},~~k \in \mathbb Z^2 \setminus \{(0,0)\},
\end{equation}
of the Stokes operator. Alternatively one starts with two independent periodic $\alpha$-regular Gaussian random fields over $\Omega$ (e.g., expanded in a basis of periodic wavelets with independent Gaussian coefficients, as in \cite{NR20}) and then applies the (linear) Helmholtz-Leray-projector $P$ to the Gaussian vector to enforce the constraint that the prior concentrates on the space $H$.  The $N$-dependent rescaling of $\theta'$ follows ideas in \cite{MNP21} and provides an (in the proofs essential) increase of the amount of regularisation provided by the prior. It corresponds to common choices of `tuning parameters' for penalised least squares (or MAP) estimators, see Remark \ref{map} below for more discussion. If information on $\theta$ is available from past observations (`training samples'), we could centre the prior at such a `trained' mean vector, but for the theory we only consider generic mean \textit{zero} Gaussian process priors.

\smallskip

As the random vector fields $\theta', \theta$ lie almost surely in the space $V \cap H^2 = \mathcal D(A)$, by Proposition \ref{classical} such a prior postulates a complete stochastic `forward' model of uniformly bounded solutions $u_\theta(t, \cdot)$ of the Navier-Stokes equations (\ref{nstokes}) drawn from the Gaussian initial condition $\theta$. Given observations $Z^{(N)}$ we can update this model via Bayes' rule to produce the best `posterior' fore- and hind-casts for the solution $u_\theta(x, t)$ at (potentially unobserved) times $t$ and points $x \in \Omega$.  Algorithmically we first compute the posterior distribution for the initial state $\theta$, which in the model (\ref{model}) is of the form
\begin{equation}\label{post}
d\Pi(\theta|Z^{(N)}) \propto e^{\ell_N(\theta)} d\Pi(\theta);~~\ell_N(\theta) = -\frac{1}{2} \sum_{i=1}^N |Y_i - u_\theta(X_i, t_i)|^2, ~ \theta \in V,
\end{equation}
where $|\cdot|=|\cdot|_{\R^2}$ is the Euclidean norm. Even though the prior is Gaussian, the non-linearity of the map $\theta \to u_\theta$ renders $\Pi(\cdot|Z^{(N)})$ a non-Gaussian random probability measure in the function space $V$. Nevertheless,
posterior draws $\theta \sim \Pi(\cdot|Z^{(N)})$ and then also an estimate for the posterior mean $E^\Pi[\theta| Z^{(N)}]$ can be calculated from Markov chain Monte Carlo (MCMC) techniques, for instance by the pCN, ULA or MALA algorithm (see \cite{CRSW13} and specifically in the context of data assimilation also Ch.3 in \cite{LSZ15}; as well as \cite{HSV14, NW20, N23} for results towards computational guarantees). This approach requires numerical solutions of the forward PDE at each iterate $\vartheta_k$ of the Markov chain -- but no inversion step, or backward solution of the PDE, is required. We can then compute an estimate $u_{\bar \theta}$ for $u_\theta(x, t)$ at any given point $(x,t)$ by computing the solution of (\ref{nstokes}) with the initial condition $\bar \theta =E^\Pi[\theta| Z^{(N)}]$ (itself approximated by ergodic MCMC averages $\sum_{k=1}^K \vartheta_k/K$ on a suitable discretisation space for $\theta$). A related optimisation based approach to compute `MAP estimates' is discussed in Remark \ref{map} below.

\smallskip

The absence of an explicit inversion step in this algorithm is attractive for applications but also triggers the question whether guarantees can be given that it will recover the true physical state of the system. Our goal here is to prove that this method indeed will be statistically \textit{consistent}, that is, that it will recover the `correct' solution of the Navier-Stokes equations at any point in time and space, arising from the `ground truth' initial condition $\theta_0$ that has actually generated the data, at least if we take sufficiently many measurements $N \to \infty$. Mathematically this means that we study the statistical behaviour of the posterior distribution under the law $P_{\theta_0}^N$, following the usual paradigm of frequentist analysis of Bayes procedures, see \cite{GV17} or also Ch.7.3 in \cite{GN16}. In our proofs we combine Proposition \ref{classical} and Theorem \ref{main1} with recent techniques from the theory of Bayesian non-linear inversion with Gaussian process priors \cite{MNP21, N23} to show asymptotic concentration properties of this posterior around the true states $(u_{\theta_0}(t,x), t \ge0, x \in \Omega)$ of the non-linear dynamical system.

\smallskip

By convention we regard the first $N$ measurements as the training sample and the additional pair $(X_{N+1}, t)$ as the prediction sample, where $X_{N+1}$ is drawn at random from $\lambda_\Omega$ and $t \in (0,T_p], T_p \ge T,$ is a (deterministic) time we wish to fore- or hind-cast. If we denote by $\theta_0$ the ground truth initial condition that generated the data (\ref{model}) and by $\theta \sim \Pi(\cdot|Z^{(N)})$ a draw from the posterior distribution, then this leads us to consider the posterior quadratic `prediction risk'
\begin{equation}\label{fwdt}
E_{X_{N+1}}\Big(\big|u_\theta(t, X_{N+1}) - u_{\theta_0}(t, X_{N+1})\big|^2\big) = \|u_\theta(t, \cdot) - u_{\theta_0}(t, \cdot)\|_{L^2(\Omega)^2}^2,~t>0,
\end{equation}
which measures how well we predict on average the state of the system at time $t$ and at a `generic' position $X_{N+1} \sim \lambda_\Omega$.

\smallskip

\begin{theorem}\label{mainstat}
Consider a Gaussian process prior as in Condition \ref{gprior} with $\alpha \ge 2$, RKHS $\mathcal H$, and resulting posterior distribution (\ref{post}) arising from observations (\ref{model}) in the $2$-dimensional periodic Navier-Stokes equations (\ref{nstokes}) with either $0 \le T_0 < T$ or $T_0=T>0$. Suppose the ground truth initial condition $\theta_0$ lies in $\mathcal H$. Then for every $T_P \ge T$ there exists a sequence $\eta_N = O(1/\sqrt {\log N})$ as $N \to \infty$ (with constants uniform in $\|\theta_0\|_\mathcal H \le U$) such that
\begin{equation}\label{darat}
\Pi\Big(\theta \in V: \sup_{0 < t \le T_p}\|u_\theta(t, \cdot) - u_{\theta_0}(t, \cdot)\|_{L^2(\Omega)^2}< \eta_N|Z^{(N)} \Big) \to^{P_{\theta_0}^N} 1
\end{equation}
as well as
\begin{equation}\label{invrat}
\Pi\big(\theta \in V:  \|\theta - \theta_0\|_{L^2(\Omega)^2} < \eta_N|Z^{(N)}\big) \to^{P_{\theta_0}^N} 1.
\end{equation}
Moreover, if $$\bar \theta_N = E^\Pi[\theta|Z^{(N)}]\in V$$ is the posterior (`Bochner-') mean and $u_{\bar \theta_N}$ the solution of the Navier-Stokes equation (\ref{nstokes}) with initial condition $\bar \theta_N$, then
\begin{equation}
\|\bar \theta_N - \theta_0\|_{L^2(\Omega)^2} + \sup_{0 < t \le T_p}\|u_{\bar \theta_N}(t, \cdot) - u_{\theta_0}(t, \cdot)\|_{L^2(\Omega)^2} =O_{P^N_{\theta_0}}(\eta_N).
\end{equation}
\end{theorem}

The logarithmic rates obtained may not be sharp at `observed' times $t>0$. For instance in `average prediction' loss where one takes the quadratic time average of (\ref{fwdt}) over $[T_0,T]$, our proofs imply much faster rates, see (\ref{fwdrate}). But inference for fixed possibly unobserved times $t$ constitutes a nonlinear inverse problem. Our proofs show that the Bayesian posterior distribution arising from a prior on the initial condition $\theta$ solves this problem `implicitly' by first inferring $\theta$ and then updating the resulting forward model prediction for $u_\theta(t)$. To avoid the inversion step one may be tempted to construct a prior for $u(t)$ at observed times only without specifying an initial condition, but due to the non-linear constraint on the `regression' functions $u(t)$ to satisfy the Navier-Stokes equations at $t>0$, this is not straightforward. In contrast, a prior on $\theta$ and its non-linear push-forward $u_\theta(t)$ can be obtained from a draw of a Gaussian process and a solution of the associated PDE with Gaussian initial values. 

We now show that for recovery of the initial condition, the posterior convergence rates obtained here are essentially optimal in an information theoretic `minimax' sense in the family of two-dimensional periodic Navier-Stokes equations arising from $H^2$-initial conditions.

\begin{theorem}\label{minmax}
Consider periodic solutions of the Navier-Stokes equations (\ref{nstokes}) with viscosity $\nu=1/2$, forcing $f=0$, and observations $Z^{(N)}$ arising as in (\ref{model}) with either $0<T_0 < T$ or $T_0=T>0$. Then there exists $c=c(U, T)>0$ such that
\begin{equation}
\liminf_{N \to \infty} ~\inf_{\tilde \theta_N}~ \sup_{\theta \in V: \|\theta\|_{H^2} \le U}  P^N_\theta \Big(\|\tilde \theta_N - \theta\|_{L^2} > \frac{c}{\log N}\Big) >1/4,
\end{equation}
where the infimum extends over all estimators $\tilde \theta_N$ of $\theta$ (i.e., all measurable functions of $Z^{(N)}$ taking values in the space $V$).
\end{theorem}

The proof uses Theorem \ref{stabsharp} and is reminiscent of similar `logarithmic' minimax rates in the \textit{linear} inverse problem of recovering the initial condition from an observed solution of the scalar heat equation (see, e.g., \cite{KVV17, R13}). As with the much simpler case of such heat equations, one can ask if faster rates can be obtained for `super-smooth' initial conditions. If we use a Gaussian prior that has a slowly growing expansion in the eigenfunctions from (\ref{stokeef}), and if the true initial condition $\theta_0$ is `band-limited' in the Stokes spectrum, then we can indeed obtain convergence rates that approach the `parametric' rate $1/\sqrt N$ of finite-dimensional models as we increase the regularity of the prior, $\alpha \to \infty$.

\begin{theorem}\label{fastrates}
Denote by $(e_j: j \in \mathbb N) \subset V$ an enumeration of the $L^2$-orthonormal basis of $H$ arising from the eigenfunctions of the Stokes operator $A$ from (\ref{stokeef}), ordered by increasing eigenvalues. Let the prior $\Pi$ be as in Condition \ref{gprior} and project it onto the span $E_J=\{e_j: j \le J\}$ with $J = J_N = O(\log \log N)$. Suppose the ground truth initial condition $\theta_0$ lies in $\mathcal H \cap E_{J_0}$ for some arbitrary fixed $J_0 \in \mathbb N$. Then the conclusions of Theorem \ref{mainstat} remain true with convergence rate $$\eta_N = (\log N)^\beta \times N^{-\alpha/(2\alpha+2)},~\text{ some }~\beta>0.$$
\end{theorem}

\begin{remark}[MAP Estimates] \normalfont \label{map}
Instead of computing posterior means via MCMC sampling one can also attempt to minimise the negative log-posterior density (\ref{post}) over a suitable approximation space of $\theta$'s. Resulting minimisers are commonly used in \textit{variational data assimilation} and are sometimes called `maximum a posteriori (MAP)' estimates. They co-incide with penalised least squares estimates where the squared RKHS norm of the prior determines the penalty functional (cf.~p.349 in \cite{GV17}). We note that in this setting, the shrinkage towards zero of the prior in Condition \ref{gprior} is natural and corresponds to the common $N$-dependent choice of the optimal scaling of the penalty used to obtain convergence rates (see, e.g., Sec.~10.1 in \cite{vdG00} in `direct' regression models, and also \cite{NvdGW20}). As $\theta \mapsto u_\theta$ is non-linear, the resulting optimisation problem is non-convex and so it is not clear that a global optimiser can be computed, but if it can, the analytical techniques developed here also allow one to give statistical guarantees for such methods. We refer to  \cite{NvdGW20} and Sec.~5.1 in \cite{BN21} where convergence rates of MAP estimates in general non-linear inverse problems are obtained. These theorems require one to verify analytical hypotheses about the associated forward maps that are slightly more involved than what is required in \cite{N23} (and used in the present contribution) to obtain convergence of the posterior measure and its mean. But in principle convergence rate results similar to Theorems \ref{mainstat} and \ref{fastrates} are provable for such MAP estimates as well by upgrading Proposition \ref{classical} appropriately and combining it with  Theorem \ref{main1}.
\end{remark}

\section{Proofs}\label{proofs}

\subsection{Proof of Theorem \ref{main1}}

The idea of the proof is based on \cite{BT73}. We can assume $w(0) \equiv u(0) - v(0) \neq 0$ in $V \subset L^2(\Omega)$ and by forward uniqueness therefore also $w(t)\equiv u(t)-v(t) \neq 0$ in $V \subset L^2(\Omega)$ for all $t$. The `Dirichlet ratio' at time $t$ is defined as
\begin{equation}\label{dirrat}
\Phi(t) = \frac{\|w(t)\|^2_{V}}{\|w(t)\|^2_{L^2}} = \frac{\langle A w(t), w(t)\rangle_{L^2}}{\|w(t)\|^2_{L^2}},~~ t \in [0,T],
\end{equation}
where we recall (\ref{divA}) and where we now write, unless specified otherwise, $L^2=L^2(\Omega)$. If we set $\bar u = (u+v)/2$ then we see from (\ref{B}) and an elementary calculation that
\begin{equation}
B(u,u) - B(v,v) = B(\bar u, w) + B(w, \bar u).
\end{equation}
Then since $u(t), v(t)$ solve the Navier-Stokes equations for the respective initial conditions we see that $w(t)$ solves the inhomogeneous non-linear parabolic equation in $H$ given by
\begin{equation}\label{pseudlin}
\frac{dw}{dt} + \nu Aw = g, ~~\text{ where } g(t)= -B(\bar u(t), w(t)) - B(w(t), \bar u(t))
\end{equation}
with initial condition $w(0)=u(0)-v(0)$. The following bounds for the Dirichlet ratio (\ref{dirrat}) will be the key to the proof of Theorem \ref{main1}.
\begin{lemma}
We have
\begin{equation}
\frac{d}{dt}\Phi(t) \le \frac{\|g(t)\|_{L^2}^2}{\nu \|w(t)\|_{L^2}^2}~~\forall t \in (0,T].
\end{equation}
Moreover,
\begin{equation}
\|g(t)\|_{L^2} \le k(t) \|w(t)\|_{V}~~\text{for all}~~ t \in (0,T]
\end{equation}
for some $k \in L^4((0,T))$ whose $L^4$-norm is bounded by a fixed constant that depends on $T, \nu, \|f\|_{L^2}$ and on the initial conditions $u(0), v(0)$ only via the upper bound $U \ge \|u(0)\|_V + \|v(0)\|_V$.
\end{lemma}
\begin{proof}
We take the $L^2$-inner product of equation (\ref{pseudlin}) with $w$ and $Aw$, respectively, which gives
$$\frac{1}{2}\frac{d}{dt} \|w(t)\|^2_{L^2} + \nu \|w(t)\|_{V}^2 = \langle g, w \rangle_{L^2},$$
$$\frac{1}{2}\frac{d}{dt} \|w(t)\|_{V}^2 + \nu \|Aw(t)\|_{L^2}^2 = \langle g, Aw \rangle_{L^2}$$
where we have also used (\ref{divA}) in the second identity. Therefore we have
\begin{align*}
\frac{d}{dt} \Phi(t) &= 2 \frac{\|w\|_{L^2}^2 \big(\langle g, Aw \rangle_{L^2} - \nu \|Aw\|_{L^2}^2 \big) - \|w\|_{V}^2\big(\langle g, w \rangle_{L^2} - \nu \|w\|_{V}^2 \big)}{\|w\|_{L^2}^4} \\
&=2 \nu \frac{\|w\|_{V}^4 - \|w\|_{V}^2 \langle g/\nu, w\rangle_{L^2} - \|w\|_{L^2}^2\|Aw\|_{L^2}^2 + \|w\|_{L^2}^2 \langle g/\nu, Aw \rangle_{L^2}}{\|w\|_{L^2}^4} \\
&=2\nu \frac{\big(\|w\|_V^2 - \langle g/2\nu, w\rangle_{L^2}\big)^2 - \langle g/2\nu, w \rangle_{L^2}^2 - \|w\|_{L^2}^2\|Aw\|_{L^2}^2 + \|w\|_{L^2}^2 \langle g/\nu, Aw \rangle_{L^2} }{\|w\|_{L^2}^4}\\
&=2\nu \frac{\langle Aw - g/2\nu, w\rangle_{L^2}^2 - \langle g/2\nu, w \rangle_{L^2}^2 - \|w\|_{L^2}^2\|Aw\|_{L^2}^2 + \|w\|_{L^2}^2 \langle g/\nu, Aw \rangle_{L^2} }{\|w\|_{L^2}^4}.
\end{align*}
using again (\ref{divA}) in the last step. By the Cauchy-Schwarz inequality, the last ratio is bounded from above by
\begin{align*}
\frac{2\nu}{\|w\|_{L^2}^4}\Big[ \big\|Aw-\frac{g}{2\nu}\big\|_{L^2}^2 \|w\|_{L^2}^2 + \frac{\|g\|_{L^2}^2 \|w\|_{L^2}^2}{4\nu^2} -  \|w\|_{L^2}^2\|Aw\|_{L^2}^2 + \|w\|_{L^2}^2 \big\langle \frac{g}{\nu}, Aw \big\rangle_{L^2}\Big] = \frac{\|g\|_{L^2}^2}{\nu \|w\|_{L^2}^2}
\end{align*}
and the first claim of the lemma follows. Using again the Cauchy-Schwarz inequality, the Poincar\'e inequality (p.292 in \cite{E10}) $$\frac{\|w(t)\|_{L^2}}{\|\nabla w(t)\|_{L^2}} \le c,$$ Ladyzhenskaya's inequalities for $L^4$-norms ((9.28) in \cite{R01}), and Lemma 4.9 in \cite{CF88} (in fact its periodic analogue), we can bound $\|g(t)\|_{L^2}$ by
\begin{align*}
&\|B(\bar u(t), w(t))\|_{L^2} +\|B(w(t), \bar u(t))\|_{L^2} \\
&\le \|w(t)\|_{L^4(\Omega)} \|\nabla \bar u(t)\|_{L^4(\Omega)} + \|\bar u(t)\|_\infty \|\nabla w(t)\|_{L^2} \\
&\le c_0\|w(t)\|^{1/2}_{L^2(\Omega)} \|\nabla w(t)\|_{L^2(\Omega)}^{1/2}\|\nabla \bar u(t)\|_{L^4(\Omega)} + \|\bar u(t)\|_\infty \|\nabla w(t)\|_{L^2(\Omega)} \\
&\le  \|w(t)\|_{V} \big(c_1\|\bar u(t)\|_{H^1}^{1/2} \|\bar u(t)\|_{H^2}^{1/2} + c_2 \|\bar u(t)\|^{1/2}_{L^2} \|\bar u(t)\|_{H^2}^{1/2} \big) \equiv k(t) \|w(t)\|_{V},
\end{align*}
with appropriate universal constants $c_i$. Now by (\ref{fwdreg}) the norms $\|\bar u(t)\|_{L^2}, \|\bar u(t)\|_{V}$ are bounded by a fixed constant and $\|\bar u(t)\|_{H^2}$ lies in $L^2((0,T))$ for strong solutions of the Navier-Stokes equations. In particular the $L^4$-norms of $k$ are bounded as required, completing the proof of the lemma.
\end{proof}

Combining the two inequalities from the last lemma we obtain the differential inequality
\begin{equation}
\frac{d}{dt} \Phi(t) \le \frac{k^2(t)}{\nu} \Phi(t)~~\forall t \in (0, T],
\end{equation}
which after integrating (i.e., Gronwall's inequality, p.711 in \cite{E10}) implies
\begin{equation}\label{initbd}
\Phi(t) \le \Phi(0) e^{\frac{1}{\nu}\int_0^t k^2(s)ds} \equiv \Phi(0) K(t) ~~\forall t \in (0, T],
\end{equation}
where we notice that $\Phi(0)$ is bounded by a constant from below by the Poincar\'e inequality (p.292 in \cite{E10}), and also finite as $w(0)\neq 0$ and $u(0), v(0) \in V$ by hypothesis.

Now taking inner products of (\ref{pseudlin}) with $w$ and noting that $\langle B(\bar u, w), w \rangle_{L^2}=0$ from (\ref{divB}) we arrive at
$$\frac{1}{2} \frac{d}{dt}\|w(t)\|_{L^2}^2 + \nu \|w(t)\|_{V}^2 + \langle B(w, \bar u), w\rangle_{L^2}=0,~~0<t \le T.$$ Using again a standard inequality (9.26) in \cite{R01} for the relevant tri-linear form, we can bound 
$$|\langle B(w, \bar u), w\rangle_{L^2(\Omega)}| \le C\|w\|_{L^2} \|w\|_{V} \|\bar u\|_{V}$$ for some $C<\infty$.  In this way we obtain for all $0<t \le T$ that
\begin{align*}
\frac{1}{2} \frac{d}{dt}\|w(t)\|_{L^2}^2 &\ge -\nu \|w(t)\|_V^2 - C \|w(t)\|_{L^2} \|w(t)\|_{V} \|\bar u(t)\|_{V} \\
& \ge \Big(-\nu \Phi(t) - C\|\bar u(t)\|_V \Phi^{1/2}(t)\Big) \|w(t)\|_{L^2}^2 \\
& \ge - \phi_T \|w(t)\|_{L^2}^2
\end{align*}
where
$$\phi_T \equiv \sup_{0 \le t \le T}[\nu \Phi(t) + C\|\bar u(t)\|_V \Phi^{1/2}(t)].
$$
But from this we deduce
\begin{equation}\label{keybd}
\|w(t)\|_{L^2}^2 \ge \|w(0)\|_{L^2}^2 e^{-2T \phi_T}~~\forall  t\in (0,T].
\end{equation}
We can also integrate the last inequality over the interval $[T_0,T]$ to obtain similarly
\begin{equation}\label{keybdint}
\frac{1}{T-T_0}\int_{T_0}^T \|w(t)\|_{L^2}^2 dt \ge \|w(0)\|_{L^2}^2 e^{-2T \phi_T},
\end{equation}
as will be relevant later, in Corollary \ref{usef} below. It remains to examine the constant $\phi_T$, and we distinguish the two cases A) and B) now.

\smallskip

B) Assume first the simpler case where an apriori bound (\ref{invpoinc}) on the `inverse Poincar\'e constant' $\Phi(0)\le c_P$ is available. Then assuming $c_P \ge 1$ without loss of generality we have from (\ref{initbd}) and Proposition \ref{classical} that $$\phi_T \le \sup_t[{\nu \Phi(t) + C\|\bar u(t)\|_V \Phi^{1/2}(t)}] \le \mathcal K c_P$$ for some constant $\mathcal K=\mathcal K(U, \|f\|_{L^2}, \nu, T)$ and so (\ref{keybd}) becomes
$$\|w(0)\|_{L^2} \le e^{c_2 c_P} \|w(t)\|_{L^2},~~c_2=2T\mathcal K,$$ which is the desired stability estimate.

\bigskip

A) We can proceed as in the last step but use the estimate $\Phi(0) \le U/ \|u(0)-v(0)\|_{L^2}$ instead of appealing to an inverse Poincar\'e constant. By the usual Poincar\'e inequality, $\Phi(0) \ge c>0$ and hence $\sqrt {\Phi(0)} \lesssim \Phi(0)$. Combining this with (\ref{initbd}) and (\ref{keybd}) gives
\begin{equation}\label{agee}
\|u(0)-v(0)\|_{L^2}^2 \exp\Big\{-\frac{c_1^2}{\|u(0)-v(0)\|_{L^2}^2}\Big\} \le \|u(t)-v(t)\|^2_{L^2},~~0<t \le T.
\end{equation}
for a constant $c_1=c_1(T,U,\nu, \|f\|_{L^2})$ that we can choose to exceed $\sup_{0 \le t \le T}\|u(t)-v(t)\|_{L^2}$ in view of Proposition \ref{classical}. Since $e^{-c_1^2/x^2} \le (x/c_1)^2$ for all $x>0$ we deduce
$$\exp\Big\{-\frac{2c_1^2}{\|u(0)-v(0)\|_{L^2}^2}\Big\} \le  \frac{\|u(t)-v(t)\|^2_{L^2}}{c_1^2} \equiv Z$$ with right hand side $Z <1$. Hence $\log Z<0$, so taking logarithms in the previous display gives
$$-\frac{2c_1^2}{\|u(0)-v(0)\|_{L^2}^2} \le -\log (1/Z) \iff  \|u(0)-v(0)\|_{L^2}^2 \le \frac{2c_1^2}{\log (1/Z)}$$
from which it follows that
$$\|u(0)-v(0)\|_{L^2} \leq \sqrt{2} c_1  \Big(\log \frac{c_1^2}{\|u(t)-v(t)\|^2_{L^2}} \Big)^{-1/2},$$ completing the proof.

\subsection{Proof of Theorem \ref{stabsharp}}

\begin{proof}
For $j =1,2,\dots$ we define univariate trigonometric polynomials on $[0,2\pi]$ as $$\phi_j(x) =  \cos(jx), x \in [0,2\pi].$$ Then for initial conditions $w_j(0)=\phi_j/j^2$ the (scalar) heat equation $$\frac{\partial w}{\partial t} -  \frac{\partial^2 w}{\partial x^2} = 0~~\text{on } [0,2\pi] \times [0,T]$$ has unique solutions $$w_j(t,x)=\frac{-e^{-j^2 t}}{j^2} \phi_j(x),~x \in [0,2\pi],~ t \in [0,T].$$  Now following an idea in Remark 7 in \cite{FS84}, consider the periodic Navier-Stokes equation (\ref{nstokes}) for initial conditions $$u_j(0,x) = j^{-2} (\phi_j(x_1-x_2), \phi_j(x_1-x_2)),~~x = (x_1,x_2) \in \Omega,$$ which clearly satisfy $\nabla \cdot u_j(0, \cdot)=0$ and hence lie in $C^\infty(\Omega) \cap V$. The vector fields $$u_j(t,x) = (w_j(t, x_1-x_2), w_j(t, x_1-x_2)), ~~x = (x_1,x_2) \in \Omega,$$ are also divergence free, have \textit{vanishing non-linear term} $(u_j \cdot \nabla) u_j=0$, and solve the Navier-Stokes equations (\ref{nstokes}) for $\nu=1/2, f=0$ and the given initial conditions. By translation invariance of Lebesgue measure
$$\int_0^{2\pi} \int_0^{2\pi} \phi^2_j(x_1-x_2) dx_1 dx_2 = 2\pi \int_0^{2\pi} \phi^2_j(x_1) dx_1 $$
from which one easily deduces
$$\|u_j(0)\|_{L^2([0,2\pi]^2)^2} \simeq j^{-2},~~ \|u_j(t)\|_{L^2([0,2\pi]^2)^2} \simeq e^{-j^2 t}/j^2.$$
The Sobolev norms of these initial conditions can also be computed directly and are of order $\|u_j(0)\|^2_{H^2} \lesssim \frac{j^2}{j^2} \le const.$ The inequality (\ref{invest}) now follows from (\ref{invest0}), $v(0)=v(t)=0$ and elementary properties of the exponential/logarithm map.
\end{proof}

\subsection{Proof of Theorems \ref{mainstat} and \ref{fastrates}}

We will apply the general theory for non-linear Bayesian inverse problems from \cite{N23} in conjunction with Proposition \ref{classical} and the stability estimate Theorem \ref{main1}. In fact (unless $T_0=T$) we will use the following corollary to this theorem, proved by just replacing (\ref{keybd}) by (\ref{keybdint}) in the last step of its proof.
\begin{corollary} \label{usef}
In the setting of the Theorem \ref{main1} B),  there exists a constant $c_2$ depending on $U, \nu, T, \|f\|_{L^2}$ such that
\begin{equation}\label{useful}
\|u(0)- v(0)\|_{L^2(\Omega)} \le e^{c_2 c_P} \Big(\frac{1}{T-T_0}\int_{T_0}^T \|u(t, \cdot) - v(t, \cdot)\|^2_{L^2(\Omega)}dt \Big)^{1/2},
\end{equation}
while in the setting of Theorem \ref{main1} A), there exist constants $c_0, c_1$ depending on $U, \nu, T, \|f\|_{L^2}$ such that
\begin{equation}\label{usefullog}
\|u(0)-v(0)\|_{L^2(\Omega)} \leq c_0 \Big(\log \frac{c_1}{\frac{1}{T-T_0}\int_{T_0}^T \|u(t, \cdot) - v(t, \cdot)\|^2_{L^2(\Omega)}dt} \Big)^{-1/2}
\end{equation}
where $\frac{1}{T-T_0}\int_{T_0}^T \|u(t, \cdot) - v(t, \cdot)\|^2_{L^2(\Omega)}dt<c_1$.
\end{corollary}

Now in the notation of Section 1.2.1 in \cite{N23}, the parameter space $\Theta$ of initial conditions is chosen as the subspace $$\Theta \equiv V \cap H^2(\Omega)^2~~\text{of } L^2(\mathcal Z, W)$$ with $V$ from (\ref{V}), and where we choose $\mathcal Z=\Omega$ and $W=\R^2$. The forward map $$\theta \mapsto \G(\theta)=u_\theta,~ \G: \Theta \to L^2(\mathcal X, \R^2),$$ is the solution map of the PDE (\ref{nstokes}) with initial condition $u(0) = \theta$ (cf.~Proposition \ref{classical}), where we set $\mathcal X = (T_0,T] \times \Omega$ and the finite-dimensional vector space $V$ in \cite{N23} (different from $V$ in our context) is identified with $\R^2$ here. Then on $\mathcal X$ we have the uniform probability measure $\lambda$ (or $\lambda = \delta_{T} \otimes \lambda_\Omega$ if $T_0=T$), while on $\mathcal Z$ we can just take $\zeta$ equal to Lebesgue measure, so that our statistical model (\ref{model}) co-incides precisely with the one in eq.~(1.9) in \cite{N23} with `random design' $(t_i,X_i)_{i=1}^N$ equal to the $(X_i)_{i=1}^N$ there. By hypothesis, the prior is supported in $V \cap \mathcal D(A) \equiv \mathcal R$ almost surely and we will apply Theorem 2.2.2 (and Theorem 1.3.2) in \cite{N23} with regularisation space $\mathcal R$, norm $\|\cdot\|_{\mathcal R} \equiv \|\cdot\|_{H^2}$, and choices $$\kappa=0, \delta_N = N^{-\alpha/(2\alpha + 2)}, ~d=2.$$  The Condition 2.1.1 in \cite{N23} is then verified in view of Proposition \ref{classical} , and we obtain that for some large enough constant $M>0$, as $N \to \infty$,
\begin{equation}\label{fwdrate}
\Pi \Big(\theta \in V:  \|\theta\|_{H^2} \le M, \|u_\theta - u_{\theta_0}\|_{L^2(\mathcal X, \lambda)} \le M \delta_N|Z^{(N)}\Big) \to^{P_{\theta_0}^N} 1
\end{equation}
where the contraction rate holds for the norm given by
\begin{align}\label{KLd}
\|v\|^2_{L^2(\mathcal X, \lambda)} & = \|v(T)\|^2_{L^2(\Omega)}~~~~~~~~~~~~~~~~~~~~~~~~~\text{ when } T_0=T, \\
&= \frac{1}{T-T_0}\int_{T_0}^T \|v(t)\|^2_{L^2(\Omega)}dt~~~\text{ when } 0 \le T_0 <T. \notag
\end{align}
By the stability estimate (\ref{usefullog}) from Corollary \ref{usef} with $u=u_\theta, v=u_{\theta_0}$ this gives contraction rate $\eta_N$ for $\|\theta-\theta_0\|_{L^2}$ and proves (\ref{invrat}) for $0 \le T_0<T$; when $T_0=T$ we use Theorem \ref{main1}A) directly instead of its corollary. The bound (\ref{darat}) then follows from what precedes and the forward estimate from Proposition \ref{main1}, Part B) (now with $T_p$ replacing $T$). The last claim of Theorem \ref{mainstat} (convergence rate of the posterior mean) follows from the preceding estimates and a  uniform integrability argument given in Theorem 2.3.2 in \cite{N23}, and again Proposition \ref{classical}, and details are left to the reader.

The proof of Theorem \ref{fastrates} follows the same pattern (cf.~also Exercise 2.4.3 in \cite{N23}), noting that the RKHS is now $E_J \cap \mathcal H$ were $E_J$ is the linear span of the Stokes-eigenfunctions up to order $J$. The preceding arguments go through since $\theta_0 \in E_J \cap \mathcal H$ for all $N$ (and thus $J$) large enough. We can then use the stronger stability estimate (\ref{useful}) with inverse Poincare constant now growing at most as $c_P = O(\log \log N)$ as explained in Remark \ref{stokspec}, so that the result follows by introducing the additional log-factor $e^{c_2 c_P} =O((\log N)^{\beta})$ in the contraction rate.

\subsection{Proof of Theorem \ref{minmax}}
We use a standard lower bound proof technique from nonparametric statistics, specifically Theorem 6.3.2 in \cite{GN16} (or since we will only use a two-hypotheses case, we can also argue as in the the proof of Theorem 2.2 in \cite{AN19}). Recalling (\ref{KLd}), the Kullback-Leibler divergence in our measurement model (\ref{model}) is $$KL(P_\theta^N, P_{\theta_0}^N) = \frac{N}{2} \|u_\theta - u_{\theta_0}\|_{L^2(\mathcal X, \lambda)}^2,~~\theta \in \Theta,$$ see Proposition 1.3.1 in \cite{N23} for a proof. Consider first the case $0<T_0 <T$. Using the base hypothesis $\theta_0=0$ with $u_{\theta_0} =0$, and alternative hypothesis $\theta_j=u_j(0)$ from Theorem \ref{stabsharp}, we can for every $\mu>0$ choose  $j=j_N= \sqrt{L \log N}$ such that $$N\int_{T_0}^T \|u_j(t)\|^2_{L^2(\Omega)}dt = \frac{N}{j^4} \int_{T_0}^T e^{-2j^2 t}dt = \frac{N}{2j^6} \big[e^{-2j^2T_0} - e^{-2j^2 T} \big] \lesssim N e^{-j^2T_0} \le \mu$$ for $L=L(\mu, T_0)$ large enough. At the same time $$\|\theta_{j_N} - \theta_0\|_{L^2} \gtrsim \frac{1}{\log N}$$ by Theorem \ref{stabsharp}, which verifies (6.100) in \cite{GN16} with $r_n \simeq 1/\log N$, and so the result follows from Theorem 6.3.2 in \cite{GN16} with $M=2$ and $\mu$ small enough (cf.~also (6.99) to obtain the `in probability' version of the lower bound). The same proof works for $T_0=T>0$ fixed without even estimating integrals over $(T_0, T]$.

\subsection{Proof of Proposition \ref{classical}}\label{appendix}

A proof of the existence of strong solutions of (\ref{nstokes}) can be found in Theorem 9.5 in \cite{R01}. It remains to prove the inequalities. Part B) can be found before eq.~(10.6) in \cite{CF88} or (9.38) in \cite{R01}, and also follows from simple variations of what follows, hence will not be repeated here. 

To establish the inequalities in part A) of the proposition, we will show
\begin{equation}\label{keyfwdbd}
\sup_{u(0) \in V: \|u(0)\|_{H^2} \le m} ~~ \sup_{0 \le t \le T}\|u(t)\|_{H^2(\Omega)} \le c
\end{equation}
for a constant $c=c(m, \|f\|_{H^1}, T, \nu)$. This implies (\ref{ubd}) by the Sobolev imbedding $H^2 \subset L^\infty$ and also gives (\ref{fwdreg}) for initial conditions bounded in $H^2$ (rather than in $H^1$). The proof of (\ref{fwdreg}) under weaker $H^1$-conditions follows from similar arguments, but we omit it here as we always have $\theta =u(0)$ bounded in $H^2$ elsewhere in the statistical results of this article.

\smallskip

We present here a formal proof of (\ref{keyfwdbd}) which can be justified rigorously by establishing it first for the solutions of the Galerkin approximation system of ODEs and then passing to the limit. This limiting procedure is in complete analogy to the arguments used for establishing existence for solutions cited above and will be left to the reader. We start the proof with the following preliminary a-priori estimate: Taking the $L^2$-inner product of (\ref{nstokes}) with $u$ and using (\ref{divA}), (\ref{divB}) we obtain
\begin{align*}
\frac{1}{2} \frac{d}{dt} \|u\|_{L^2}^2 + \nu \|\nabla u\|_{L^2}^2 &= \langle f, u \rangle_{L^2}  \le \lambda^{-1} \|f\|_{L^2} \|\nabla u\|_{L^2} \\
& \le \frac{1}{2 \nu \lambda} \|f\|^2_{L^2} + \frac{\nu}{2} \|\nabla u\|_{L^2}^2
\end{align*}
where we have used the Cauchy-Schwarz, the Poincar\'e (with constant $\lambda$) and the Young inequalities. This readily implies,
\begin{equation*}
\frac{d}{dt} \|u\|_{L^2}^2 + \nu  \|\nabla u\|_{L^2}^2 \le \frac{\|f\|_{L^2}^2}{\nu \lambda},
\end{equation*}
which can be integrated to give
\begin{equation}\label{e27}
\int_0^T \|\nabla u(t)\|_{L^2}^2 dt \le \frac{\|u(0)\|_{L^2}^2}{\nu} + T \frac{\|f\|_{L^2}^2}{\nu^2 \lambda} \le K_1 \equiv K_1(U,\nu, T, \|f\|_{L^2}),
\end{equation}
which we shall use below. 

Now defining the vorticity $\omega = \nabla^\perp \cdot u$ for strong solutions $u \in V$ of (\ref{nstokes}) we see that  $\|\nabla \omega \|_{L^2} = \|\Delta u\|_{L^2}$. Therefore, we can prove (\ref{keyfwdbd}) by bounding $\|\nabla \omega \|_{L^2}$ uniformly in time $t \in [0,T]$. Applying the ($2D$) curl operation $\nabla \times$ to (\ref{nstokp}) one obtains the \textit{vorticity} formulation of the Navier-Stokes equations
\begin{equation}\label{E2.5}
\frac{\partial}{\partial t} \omega - \nu \Delta \omega + (u \cdot \nabla) \omega = \nabla \times f.
\end{equation}
Taking the $L^2$-inner product of the last equation with $-\Delta \omega$ and using (\ref{divA}) we see that $\omega = \nabla^\perp \cdot u$ verifies the equation
\begin{equation}\label{E6}
\frac{1}{2} \frac{d}{dt} \|\nabla \omega\|_{L^2}^2 + \nu \|\Delta \omega\|_{L^2}^2 - \int_\Omega [(u \cdot \nabla) \omega]\Delta \omega = -\int_\Omega (\nabla \times f) \Delta \omega.
\end{equation}
 The Cauchy-Schwarz inequality implies $$\Big |-\int_\Omega (\nabla \times f) \Delta \omega\Big| \le \|\nabla f\|_{L^2} \|\Delta \omega\|_{L^2}$$ and we can further estimate, by integration by parts
\begin{align*}
\Big | -\int_\Omega [(u \cdot \nabla) \omega]\Delta \omega \Big| &= \Big| - \sum_{l=1}^2 \int_\Omega [(u \cdot \nabla) \omega] \partial_l^2 \omega \Big| = \Big| \sum_{l=1}^2 \int_\Omega [(\partial_l u \cdot \nabla \omega)\partial_l \omega \Big| \\
&\le \|\nabla u\|_{L^2} \|\nabla \omega\|_{L^4}^2 \le c_0 \|\nabla u\|_{L^2} \|\nabla \omega\|_{L^2} \|\Delta \omega\|_{L^2}
\end{align*}
since $\int_\Omega [(u \cdot \nabla) \partial_l \omega)] \partial_l \omega =0$ as in (\ref{divB}), and where we have used, for every $\phi \in H^1$ with ~$\int_\Omega \phi =0$, the Ladyzhenskaya inequality $\|\phi\|_{L^4}^2 \leq c_0 \|\phi\|_{L^2} \|\nabla \phi\|_{L^2}$ which follows from (9.28) in \cite{R01} and the Poincar\'e inequality. Substituting the previous bounds into (\ref{E6}) gives
\begin{equation*}
\frac{1}{2} \frac{d}{dt} \|\nabla \omega\|_{L^2}^2 + \nu \|\Delta \omega\|_{L^2}^2 \le  \|\nabla f\|_{L^2} \|\Delta \omega\|_{L^2} + c_0 \|\nabla u\|_{L^2} \|\nabla \omega\|_{L^2} \|\Delta \omega\|_{L^2}.
\end{equation*}
By Young's inequality for products we deduce
\begin{equation*}
\frac{1}{2} \frac{d}{dt} \|\nabla \omega\|_{L^2}^2 + \nu \|\Delta \omega\|_{L^2}^2 \le \frac{\|\nabla f\|_{L^2}^2}{\nu} +\frac{\nu}{4}\|\Delta \omega\|^2_{L^2} + \frac{c_0^2}{\nu} \|\nabla u\|^2_{L^2} \|\nabla \omega\|_{L^2}^2 + \frac{\nu}{4} \|\Delta \omega\|_{L^2}^2
\end{equation*}
and rearranging we obtain
\begin{equation*}
\frac{d}{dt} \|\nabla \omega\|_{L^2}^2 \le \frac{2}{\nu}\|\nabla f\|_{L^2}^2 + \frac{2c_0^2}{\nu} \|\nabla u\|^2_{L^2} \|\nabla \omega\|_{L^2}^2.
\end{equation*}
We can apply Gronwall's inequality (p.711 in \cite{E10}) to deduce
\begin{equation}\label{gronw}
\|\nabla \omega(t)\|_{L^2}^2 \le  e^{\frac{2c_0^2}{\nu}\int_0^t  \|\nabla u(s)\|^2_{L^2}ds} \Big[\|\nabla \omega(0)\|_{L^2}^2  + \frac{2t}{\nu} \|\nabla f\|_{L^2}^2\Big],~~ 0 < t \le T.
\end{equation}
Now we use (\ref{e27}) to bound the constants in the preceding integral and deduce
\begin{equation}
\|u\|_{H^2} \lesssim \|\Delta u(t)\|_{L^2}^2 = \|\nabla \omega(t)\|_{L^2}^2 \lesssim  \|\Delta u(0)\|^2_{L^2} + \|f\|_{H^1}^2 \le K_2,~~ 0 \le t \le T,
\end{equation}
for $K_2=K_2(T, \nu, \|f\|_{H^1}, m)$, which implies (\ref{keyfwdbd}) as desired.

\bigskip

\textbf{Acknowledgements.} The authors are grateful to the associate editor and three anonymous referees for their many helpful remarks and suggestions in the revision process. RN was supported by an ERC Advanced Grant (Horizon Europe UKRI G116786) as well as by EPSRC grant EP/V026259. The research of EST was made possible by NPRP grant \#S-0207-200290 from the Qatar National Research Fund (a member of Qatar Foundation), and is based upon work supported by King Abdullah University of Science and Technology (KAUST) Office of Sponsored Research (OSR) under Award No. OSR-2020-CRG9-4336. The work of EST has also benefited from the inspiring environment of the CRC 1114 ``Scaling Cascades in Complex Systems'', Project Number 235221301, Project A02, funded by Deutsche Forschungsgemeinschaft (DFG).

 \bibliography{navstok}{}

\begin{thebibliography}{42}

\bibitem{AN19}
\begin{barticle}[author]
\bauthor{\bsnm{Abraham},~\bfnm{Kweku}\binits{K.}} \AND
  \bauthor{\bsnm{Nickl},~\bfnm{Richard}\binits{R.}}
(\byear{2019}).
\btitle{On statistical {C}alder\'{o}n problems}.
\bjournal{Math. Stat. Learn.}
\bvolume{2}
\bpages{165--216}.
\bmrnumber{4130599}
\end{barticle}
\endbibitem

\bibitem{BT73}
\begin{barticle}[author]
\bauthor{\bsnm{Bardos},~\bfnm{C.}\binits{C.}} \AND
  \bauthor{\bsnm{Tartar},~\bfnm{L.}\binits{L.}}
(\byear{1973}).
\btitle{Sur l'unicit\'{e} r\'{e}trograde des \'{e}quations paraboliques et
  quelques questions voisines}.
\bjournal{Arch. Rational Mech. Anal.}
\bvolume{50}
\bpages{10--25}.
\bdoi{10.1007/BF00251291}
\bmrnumber{338517}
\end{barticle}
\endbibitem

\bibitem{B02}
\begin{bbook}[author]
\bauthor{\bsnm{Bennett},~\bfnm{Andrew~F.}\binits{A.~F.}}
(\byear{2002}).
\btitle{Inverse modeling of the ocean and atmosphere}.
\bpublisher{Cambridge University Press, Cambridge}.
\bdoi{10.1017/CBO9780511535895}
\bmrnumber{1920432}
\end{bbook}
\endbibitem

\bibitem{BGLFS17}
\begin{barticle}[author]
\bauthor{\bsnm{Beskos},~\bfnm{Alexandros}\binits{A.}},
  \bauthor{\bsnm{Girolami},~\bfnm{Mark}\binits{M.}},
  \bauthor{\bsnm{Lan},~\bfnm{Shiwei}\binits{S.}},
  \bauthor{\bsnm{Farrell},~\bfnm{Patrick~E.}\binits{P.~E.}} \AND
  \bauthor{\bsnm{Stuart},~\bfnm{Andrew~M.}\binits{A.~M.}}
(\byear{2017}).
\btitle{Geometric {MCMC} for infinite-dimensional inverse problems}.
\bjournal{J. Comput. Phys.}
\bvolume{335}
\bpages{327--351}.
\end{barticle}
\endbibitem

\bibitem{B21}
\begin{barticle}[author]
\bauthor{\bsnm{Bohr},~\bfnm{Jan}\binits{J.}}
(\byear{2021}).
\btitle{Stability of the non-abelian {$X$}-ray transform in dimension
  {$\geq3$}}.
\bjournal{J. Geom. Anal.}
\bvolume{31}.
\end{barticle}
\endbibitem

\bibitem{BN21}
\begin{barticle}[author]
\bauthor{\bsnm{Bohr},~\bfnm{Jan}\binits{J.}} \AND
  \bauthor{\bsnm{Nickl},~\bfnm{Richard}\binits{R.}}
(\byear{2023}).
\btitle{On log-concave approximations of high-dimensional posterior measures
  and stability properties in non-linear inverse problems}.
\bjournal{Ann. Inst. H. Poincar\'e (Probab. Stat.), to appear}.
\end{barticle}
\endbibitem

\bibitem{BGK20}
\begin{barticle}[author]
\bauthor{\bsnm{Borggaard},~\bfnm{Jeff}\binits{J.}},
  \bauthor{\bsnm{Glatt-Holtz},~\bfnm{Nathan}\binits{N.}} \AND
  \bauthor{\bsnm{Krometis},~\bfnm{Justin}\binits{J.}}
(\byear{2020}).
\btitle{On {B}ayesian consistency for flows observed through a passive scalar}.
\bjournal{Ann. Appl. Probab.}
\bvolume{30}
\bpages{1762--1783}.
\end{barticle}
\endbibitem

\bibitem{CF88}
\begin{bbook}[author]
\bauthor{\bsnm{Constantin},~\bfnm{Peter}\binits{P.}} \AND
  \bauthor{\bsnm{Foias},~\bfnm{Ciprian}\binits{C.}}
(\byear{1988}).
\btitle{Navier-{S}tokes equations}.
\bseries{Chicago Lectures in Mathematics}.
\bpublisher{University of Chicago Press, Chicago, IL}.
\bmrnumber{972259}
\end{bbook}
\endbibitem

\bibitem{CDRS09}
\begin{barticle}[author]
\bauthor{\bsnm{Cotter},~\bfnm{S.~L.}\binits{S.~L.}},
  \bauthor{\bsnm{Dashti},~\bfnm{M.}\binits{M.}},
  \bauthor{\bsnm{Robinson},~\bfnm{J.~C.}\binits{J.~C.}} \AND
  \bauthor{\bsnm{Stuart},~\bfnm{A.~M.}\binits{A.~M.}}
(\byear{2009}).
\btitle{Bayesian inverse problems for functions and applications to fluid
  mechanics}.
\bjournal{Inverse Problems}
\bvolume{25}
\bpages{115008, 43}.
\bdoi{10.1088/0266-5611/25/11/115008}
\bmrnumber{2558668}
\end{barticle}
\endbibitem

\bibitem{CRSW13}
\begin{barticle}[author]
\bauthor{\bsnm{Cotter},~\bfnm{S.~L.}\binits{S.~L.}},
  \bauthor{\bsnm{Roberts},~\bfnm{G.~O.}\binits{G.~O.}},
  \bauthor{\bsnm{Stuart},~\bfnm{A.~M.}\binits{A.~M.}} \AND
  \bauthor{\bsnm{White},~\bfnm{D.}\binits{D.}}
(\byear{2013}).
\btitle{M{CMC} methods for functions: modifying old algorithms to make them
  faster}.
\bjournal{Statist. Sci.}
\bvolume{28}
\bpages{424--446}.
\bdoi{10.1214/13-STS421}
\bmrnumber{3135540}
\end{barticle}
\endbibitem

\bibitem{CLM16}
\begin{barticle}[author]
\bauthor{\bsnm{Cui},~\bfnm{Tiangang}\binits{T.}},
  \bauthor{\bsnm{Law},~\bfnm{Kody J.~H.}\binits{K.~J.~H.}} \AND
  \bauthor{\bsnm{Marzouk},~\bfnm{Youssef~M.}\binits{Y.~M.}}
(\byear{2016}).
\btitle{Dimension-independent likelihood-informed {MCMC}}.
\bjournal{J. Comput. Phys.}
\bvolume{304}
\bpages{109--137}.
\end{barticle}
\endbibitem

\bibitem{E10}
\begin{bbook}[author]
\bauthor{\bsnm{Evans},~\bfnm{Lawrence~C.}\binits{L.~C.}}
(\byear{2010}).
\btitle{Partial differential equations},
\bedition{second} ed.
\bseries{Graduate Studies in Mathematics}
\bvolume{19}.
\bpublisher{American Mathematical Society, Providence, RI}.
\bmrnumber{2597943}
\end{bbook}
\endbibitem

\bibitem{EVvL22}
\begin{bbook}[author]
\bauthor{\bsnm{Evensen},~\bfnm{G.}\binits{G.}},
  \bauthor{\bsnm{Vossepoel},~\bfnm{F.~C.}\binits{F.~C.}} \AND
  \bauthor{\bparticle{van} \bsnm{Leeuwen},~\bfnm{J.}\binits{J.}}
(\byear{2022}).
\btitle{Data assimilation fundamentals}.
\bseries{Springer textbooks in Earth sciences, Geography and Environment}.
\bpublisher{Springer, Cham}.
\end{bbook}
\endbibitem

\bibitem{FS84}
\begin{barticle}[author]
\bauthor{\bsnm{Foias},~\bfnm{C.}\binits{C.}} \AND
  \bauthor{\bsnm{Saut},~\bfnm{J.~C.}\binits{J.~C.}}
(\byear{1984}).
\btitle{Asymptotic behavior, as {$t\rightarrow +\infty $}, of solutions of
  {N}avier-{S}tokes equations and nonlinear spectral manifolds}.
\bjournal{Indiana Univ. Math. J.}
\bvolume{33}
\bpages{459--477}.
\bdoi{10.1512/iumj.1984.33.33025}
\bmrnumber{740960}
\end{barticle}
\endbibitem

\bibitem{GV17}
\begin{bbook}[author]
\bauthor{\bsnm{Ghosal},~\bfnm{Subhashis}\binits{S.}} \AND
  \bauthor{\bparticle{van~der} \bsnm{Vaart},~\bfnm{Aad~W.}\binits{A.~W.}}
(\byear{2017}).
\btitle{Fundamentals of Nonparametric Bayesian Inference}.
\bpublisher{Cambridge University Press, New York}.
\end{bbook}
\endbibitem

\bibitem{GN16}
\begin{bbook}[author]
\bauthor{\bsnm{Gin\'e},~\bfnm{Evarist}\binits{E.}} \AND
  \bauthor{\bsnm{Nickl},~\bfnm{Richard}\binits{R.}}
(\byear{2016}).
\btitle{Mathematical foundations of infinite-dimensional statistical models}.
\bpublisher{Cambridge University Press, New York}.
\bdoi{10.1017/CBO9781107337862}
\bmrnumber{3588285}
\end{bbook}
\endbibitem

\bibitem{GN20}
\begin{barticle}[author]
\bauthor{\bsnm{Giordano},~\bfnm{Matteo}\binits{M.}} \AND
  \bauthor{\bsnm{Nickl},~\bfnm{Richard}\binits{R.}}
(\byear{2020}).
\btitle{Consistency of {B}ayesian inference with {G}aussian process priors in
  an elliptic inverse problem}.
\bjournal{Inverse Problems}
\bvolume{36}.
\end{barticle}
\endbibitem

\bibitem{GR22}
\begin{barticle}[author]
\bauthor{\bsnm{Giordano},~\bfnm{Matteo}\binits{M.}} \AND
  \bauthor{\bsnm{Ray},~\bfnm{Kolyan}\binits{K.}}
(\byear{2022, to appear}).
\btitle{Nonparametric Bayesian inference for reversible multi-dimensional
  diffusions}.
\bjournal{Ann. Statist.}
\end{barticle}
\endbibitem

\bibitem{HSV14}
\begin{barticle}[author]
\bauthor{\bsnm{Hairer},~\bfnm{Martin}\binits{M.}},
  \bauthor{\bsnm{Stuart},~\bfnm{Andrew~M.}\binits{A.~M.}} \AND
  \bauthor{\bsnm{Vollmer},~\bfnm{Sebastian~J.}\binits{S.~J.}}
(\byear{2014}).
\btitle{Spectral gaps for a {M}etropolis-{H}astings algorithm in infinite
  dimensions}.
\bjournal{Ann. Appl. Probab.}
\bvolume{24}
\bpages{2455--2490}.
\bdoi{10.1214/13-AAP982}
\bmrnumber{3262508}
\end{barticle}
\endbibitem

\bibitem{HR22}
\begin{barticle}[author]
\bauthor{\bsnm{Hoffmann},~\bfnm{Marc}\binits{M.}} \AND
  \bauthor{\bsnm{Ray},~\bfnm{Kolyan}\binits{K.}}
(\byear{2022}).
\btitle{Bayesian estimation in a multidimensional diffusion model with high
  frequency data}.
\bjournal{arXiv}.
\end{barticle}
\endbibitem

\bibitem{K03}
\begin{bbook}[author]
\bauthor{\bsnm{Kalnay},~\bfnm{Eugenia}\binits{E.}}
(\byear{2003}).
\btitle{Atmospheric modelling, data assimilation, and predictability}.
\bpublisher{Cambridge University Press, Cambridge, UK}.
\end{bbook}
\endbibitem

\bibitem{K21}
\begin{barticle}[author]
\bauthor{\bsnm{Kekkonen},~\bfnm{Hanne}\binits{H.}}
(\byear{2021}).
\btitle{Consistency of Bayesian inference with Gaussian process priors for a
  parabolic inverse problem}.
\bjournal{Inverse Problems}.
\end{barticle}
\endbibitem

\bibitem{KvdVvZ11}
\begin{barticle}[author]
\bauthor{\bsnm{Knapik},~\bfnm{Bartek}\binits{B.}}, \bauthor{\bparticle{van~der}
  \bsnm{Vaart},~\bfnm{Aad~W.}\binits{A.~W.}} \AND \bauthor{\bparticle{van}
  \bsnm{Zanten},~\bfnm{J.~Harry}\binits{J.~H.}}
(\byear{2011}).
\btitle{Bayesian inverse problems with {G}aussian priors}.
\bjournal{Ann. Statist.}
\bvolume{39}
\bpages{2626--2657}.
\bdoi{10.1214/11-AOS920}
\bmrnumber{2906881}
\end{barticle}
\endbibitem

\bibitem{KVV17}
\begin{barticle}[author]
\bauthor{\bsnm{Knapik},~\bfnm{B.~T.}\binits{B.~T.}},
  \bauthor{\bparticle{van~der} \bsnm{Vaart},~\bfnm{A.~W.}\binits{A.~W.}} \AND
  \bauthor{\bparticle{van} \bsnm{Zanten},~\bfnm{J.~H.}\binits{J.~H.}}
(\byear{2013}).
\btitle{Bayesian recovery of the initial condition for the heat equation}.
\bjournal{Comm. Statist. Theory Methods}
\bvolume{42}
\bpages{1294--1313}.
\bdoi{10.1080/03610926.2012.681417}
\bmrnumber{3031282}
\end{barticle}
\endbibitem

\bibitem{LSZ15}
\begin{bbook}[author]
\bauthor{\bsnm{Law},~\bfnm{Kody}\binits{K.}},
  \bauthor{\bsnm{Stuart},~\bfnm{Andrew}\binits{A.}} \AND
  \bauthor{\bsnm{Zygalakis},~\bfnm{Konstantinos}\binits{K.}}
(\byear{2015}).
\btitle{Data assimilation}.
\bseries{Texts in Applied Mathematics}
\bvolume{62}.
\bpublisher{Springer, Cham}
\bnote{A mathematical introduction}.
\bdoi{10.1007/978-3-319-20325-6}
\bmrnumber{3363508}
\end{bbook}
\endbibitem

\bibitem{L63}
\begin{barticle}[author]
\bauthor{\bsnm{Lorenz},~\bfnm{Edward~N.}\binits{E.~N.}}
(\byear{1963}).
\btitle{Deterministic nonperiodic flow}.
\bjournal{J. Atmospheric Sci.}
\bvolume{20}
\bpages{130--141}.
\bdoi{10.1175/1520-0469(1963)020<0130:DNF>2.0.CO;2}
\bmrnumber{4021434}
\end{barticle}
\endbibitem

\bibitem{MH12}
\begin{bbook}[author]
\bauthor{\bsnm{Majda},~\bfnm{Andrew~J.}\binits{A.~J.}} \AND
  \bauthor{\bsnm{Harlim},~\bfnm{John}\binits{J.}}
(\byear{2012}).
\btitle{Filtering complex turbulent systems}.
\bpublisher{Cambridge University Press, Cambridge}.
\bdoi{10.1017/CBO9781139061308}
\bmrnumber{2934167}
\end{bbook}
\endbibitem

\bibitem{MNP21}
\begin{barticle}[author]
\bauthor{\bsnm{Monard},~\bfnm{Fran\c{c}ois}\binits{F.}},
  \bauthor{\bsnm{Nickl},~\bfnm{Richard}\binits{R.}} \AND
  \bauthor{\bsnm{Paternain},~\bfnm{Gabriel~P.}\binits{G.~P.}}
(\byear{2021}).
\btitle{Consistent inversion of noisy non-{A}belian {X}-ray transforms}.
\bjournal{Comm. Pure Appl. Math.}
\bvolume{74}
\bpages{1045--1099}.
\end{barticle}
\endbibitem

\bibitem{MNP21a}
\begin{barticle}[author]
\bauthor{\bsnm{Monard},~\bfnm{Fran\c{c}ois}\binits{F.}},
  \bauthor{\bsnm{Nickl},~\bfnm{Richard}\binits{R.}} \AND
  \bauthor{\bsnm{Paternain},~\bfnm{Gabriel~P.}\binits{G.~P.}}
(\byear{2021}).
\btitle{Statistical guarantees for {B}ayesian uncertainty quantification in
  nonlinear inverse problems with {G}aussian process priors}.
\bjournal{Ann. Statist.}
\bvolume{49}
\bpages{3255--3298}.
\end{barticle}
\endbibitem

\bibitem{N20}
\begin{barticle}[author]
\bauthor{\bsnm{Nickl},~\bfnm{Richard}\binits{R.}}
(\byear{2020}).
\btitle{Bernstein--von {M}ises theorems for statistical inverse problems {I}:
  {S}chr\"{o}dinger equation}.
\bjournal{J. Eur. Math. Soc. (JEMS)}
\bvolume{22}
\bpages{2697--2750}.
\bdoi{10.4171/JEMS/975}
\bmrnumber{4118619}
\end{barticle}
\endbibitem

\bibitem{N22a}
\begin{barticle}[author]
\bauthor{\bsnm{Nickl},~\bfnm{Richard}\binits{R.}}
(\byear{2022}).
\btitle{Inference for diffusions from low frequency measurements}.
\bjournal{Annals of Statistics (to appear)}.
\end{barticle}
\endbibitem

\bibitem{N23}
\begin{bbook}[author]
\bauthor{\bsnm{Nickl},~\bfnm{Richard}\binits{R.}}
(\byear{2023}).
\btitle{Bayesian non-linear statistical inverse problems}.
\bseries{Zurich Lectures in Advanced Mathematics}.
\bpublisher{European Mathematical Society (EMS) press, Berlin}.
\end{bbook}
\endbibitem

\bibitem{NR20}
\begin{barticle}[author]
\bauthor{\bsnm{Nickl},~\bfnm{Richard}\binits{R.}} \AND
  \bauthor{\bsnm{Ray},~\bfnm{Kolyan}\binits{K.}}
(\byear{2020}).
\btitle{Nonparametric statistical inference for drift vector fields of
  multi-dimensional diffusions}.
\bjournal{Ann. Statist.}
\bvolume{48}
\bpages{1383--1408}.
\bdoi{10.1214/19-AOS1851}
\bmrnumber{4124327}
\end{barticle}
\endbibitem

\bibitem{NvdGW20}
\begin{barticle}[author]
\bauthor{\bsnm{Nickl},~\bfnm{Richard}\binits{R.}}, \bauthor{\bparticle{van~de}
  \bsnm{Geer},~\bfnm{Sara}\binits{S.}} \AND
  \bauthor{\bsnm{Wang},~\bfnm{Sven}\binits{S.}}
(\byear{2020}).
\btitle{Convergence rates for penalized least squares estimators in {PDE}
  constrained regression problems}.
\bjournal{SIAM/ASA J. Uncertain. Quantif.}
\bvolume{8}
\bpages{374--413}.
\bdoi{10.1137/18M1236137}
\bmrnumber{4074017}
\end{barticle}
\endbibitem

\bibitem{NW20}
\begin{barticle}[author]
\bauthor{\bsnm{Nickl},~\bfnm{Richard}\binits{R.}} \AND
  \bauthor{\bsnm{Wang},~\bfnm{Sven}\binits{S.}}
(\byear{2022}).
\btitle{On polynomial-time computation of high-dimensional posterior measures
  by {L}angevin-type algorithms}.
\bjournal{J. Eur. Math. Soc. (JEMS), to appear}.
\end{barticle}
\endbibitem

\bibitem{R13}
\begin{barticle}[author]
\bauthor{\bsnm{Ray},~\bfnm{Kolyan}\binits{K.}}
(\byear{2013}).
\btitle{Bayesian inverse problems with non-conjugate priors}.
\bjournal{Electron. J. Stat.}
\bvolume{7}
\bpages{2516--2549}.
\end{barticle}
\endbibitem

\bibitem{RC15}
\begin{bbook}[author]
\bauthor{\bsnm{Reich},~\bfnm{Sebastian}\binits{S.}} \AND
  \bauthor{\bsnm{Cotter},~\bfnm{Colin}\binits{C.}}
(\byear{2015}).
\btitle{Probabilistic forecasting and {B}ayesian data assimilation}.
\bpublisher{Cambridge University Press, New York}.
\bdoi{10.1017/CBO9781107706804}
\bmrnumber{3242790}
\end{bbook}
\endbibitem

\bibitem{R01}
\begin{bbook}[author]
\bauthor{\bsnm{Robinson},~\bfnm{James~C.}\binits{J.~C.}}
(\byear{2001}).
\btitle{Infinite-dimensional dynamical systems}.
\bseries{Cambridge Texts in Applied Mathematics}.
\bpublisher{Cambridge University Press, Cambridge}
\bnote{An introduction to dissipative parabolic PDEs and the theory of global
  attractors}.
\bdoi{10.1007/978-94-010-0732-0}
\bmrnumber{1881888}
\end{bbook}
\endbibitem

\bibitem{SS12}
\begin{barticle}[author]
\bauthor{\bsnm{Schwab},~\bfnm{C.}\binits{C.}} \AND
  \bauthor{\bsnm{Stuart},~\bfnm{A.~M.}\binits{A.~M.}}
(\byear{2012}).
\btitle{Sparse deterministic approximation of {B}ayesian inverse problems}.
\bjournal{Inverse Problems}
\bvolume{28}
\bpages{045003, 32}.
\bdoi{10.1088/0266-5611/28/4/045003}
\bmrnumber{2903278}
\end{barticle}
\endbibitem

\bibitem{StA22}
\begin{barticle}[author]
\bauthor{\bsnm{St-Amant},~\bfnm{Simon}\binits{S.}}
(\byear{2022}).
\btitle{Stability estimate for the broken non-abelian {X}-ray transform in
  {M}inkowski space}.
\bjournal{Inverse Problems}
\bvolume{38}
\bpages{Paper No. 105007, 36}.
\bmrnumber{4482473}
\end{barticle}
\endbibitem

\bibitem{S10}
\begin{barticle}[author]
\bauthor{\bsnm{Stuart},~\bfnm{Andrew~M.}\binits{A.~M.}}
(\byear{2010}).
\btitle{Inverse problems: a {B}ayesian perspective}.
\bjournal{Acta Numer.}
\bvolume{19}
\bpages{451--559}.
\end{barticle}
\endbibitem

\bibitem{vdG00}
\begin{bbook}[author]
\bauthor{\bparticle{van~de} \bsnm{Geer},~\bfnm{Sara~A.}\binits{S.~A.}}
(\byear{2000}).
\btitle{Applications of empirical process theory}.
\bpublisher{Cambridge University Press, Cambridge}.
\end{bbook}
\endbibitem

\end{thebibliography}
\bibliographystyle{imsart-number}

\end{document}